\documentclass[leqno,12pt]{article}
\usepackage{inputenc}
\usepackage[english]{babel}
\usepackage{amsmath}
\usepackage{amssymb}
\usepackage{amsthm}
\usepackage[all]{xy}
\date{}
\newtheorem{thm}{Theorem}[section]
\newtheorem{lem}[thm]{Lemma}
\newtheorem{prop}[thm]{Proposition}
\newtheorem{cor}[thm]{Corollary}

\theoremstyle{definition}
\newtheorem{ex}[thm]{\it Example}

\newtheorem{rem}[thm]{\it Remark}

\numberwithin{equation}{section}

\title{The numerical duplication of a numerical semigroup}

\begin{document}
\newcommand{\Hi}{\mathrm {Hilb}}
\newcommand{\card}{\mathrm {card}}
\newcommand{\ap}{\mathrm {Ap}}
\newcommand{\ord}{\mathrm {ord}}
\newcommand{\mapM}{\mathrm {maxAp_M}}
\newcommand{\map}{\mathrm {maxAp}}
\newcommand{\gr}{\mathrm{gr}}
\author{M. D'Anna\thanks{{\em email}: mdanna@dmi.unict.it} \and
F. Strazzanti \thanks{{\em email}:
francesco.strazzanti@gmail.com}}

\maketitle

\begin{abstract}
\noindent In this paper we present and study the numerical
duplication of a numerical semigroup, a construction that,
starting with a numerical semigroup $S$ and a semigroup ideal
$E\subseteq S$, produces a new numerical semigroup, denoted by
$S\!\Join^b\!\!E$ (where $b$ is any odd integer belonging to $S$),
such that $S=(S\!\Join^b\!\!E)/2$. In particular, we characterize
the ideals $E$ such that $S\!\Join^b\!\!E$ is almost symmetric and
we determine its type.
\medskip

\end{abstract}

\section*{Introduction}
In this paper we present and study a construction that, starting
with a numerical semigroup $S$ and a semigroup ideal $E\subseteq
S$, produces a new numerical semigroup, denoted $S\!\Join^b\!\!E$
(where $b$ is any odd integer belonging to $S$), such that $S=\{x
\in \mathbb N : \ 2x \in S\!\Join^b\!\!E \}$ (briefly
$S=(S\!\Join^b\!\!E)/2$). We call this new semigroup the {\it
numerical duplication} of $S$ with respect to $E$ and $b$.

The origin of this construction comes from ring theory; more
precisely, in \cite{DR}, the authors, looking for a unified
approach for Nagata's idealization (see e.g. Nagata's book
\cite[p.2]{Na} and \cite{fgr}) and amalgamated duplication (see
e.g \cite{D} and \cite{DF}), define a family of rings obtained as
quotients of the Rees algebra associated to a commutative ring $R$
and an ideal $I \subseteq R$; moreover, starting from an algebroid
branch $R$ and a proper ideal $I$, it is possible to obtain a
member of this family that is again an algebroid branch. It turns
out that its value semigroup is the numerical duplication of the
value semigroup $v(R)$ with respect to $v(I)$, for some odd $b \in
v(R)$.

Moreover, some homological properties, such as Gorensteinnes and
Cohen-Macaulay type, coincide for the idealization, the
duplication and for any member of the family cited above,
depending only by $R$ and $I$ (these facts are proved in the work
in progress \cite{DS}). Hence, it is natural to look at the
analogous properties for numerical semigroups. For example, in
this paper we show that if $E$ is a canonical ideal of $S$, then
$S\!\Join^b\!\!E$ is symmetric and, more generally, we determine
the type of $S\!\Join^b\!\!E$. Like in the ring case, these
properties of the numerical duplication depend only by $S$ and
$E$, but are independent by the integer $b$.

In particular, we are interested in studying when our construction
produces almost symmetric numerical semigroups. This class of
semigroups was introduced and studied by Barucci and Fr\"oberg in \cite{BF}, as the
nume\-rical analogue of the notion of almost Gorenstein rings for
one-dimensional ring theory. Almost symmetric numerical semigroups
generalize the notions of symmetric and pseudo-symmetric
semigroups (which are exactly almost symmetric numerical
semigroups of type $1$ and $2$, respectively) and revealed to be
interesting from many point of view (see e.g. \cite{B} and
\cite{N}). In \cite{GMP}, Goto, Matsuoka and Phuong studied, in particular, when
Nagata's idealization produces an almost Gorenstein ring; the
second author of the present paper studied the same problem for
the duplication in his master thesis (see \cite{S}). In this paper
we characterize those ideals $E$ of $S$ for which
$S\!\Join^b\!\!E$ is almost symmetric, obtaining a numerical
analogue of the results of \cite{GMP}.

Numerical duplication is also connected to other results in
numerical semigroup theory, regarding the notion of one half of a
numerical semigroup. In \cite{RGSGU}, Rosales, Garc\'ia-S\'anchez,
Garc\'ia-Garc\'ia and Urbano-Blanco, given a numerical semigroup
$T$ and an integer $n$, define $T/n=\{x \in \mathbb N: \ nx \in
T\}$, in order to solve proportionally modular diophantine
inequalities. Successively, in \cite{RGS2}, Rosales and
Garc\'ia-S\'anchez proved that every numerical semigroup is one
half of infinitely many symmetric numerical semigroups (and this
result has been generalized by Swanson in \cite{Sw}, for every $n
\geq 2$). As a byproduct of our characterization, we obtain that
every numerical semigroup $S$, is one half of infinitely many
almost symmetric semigroups with type $x$, where $x$ is any odd
integer not bigger than $2t(S)+1$ (here $t(S)$ denotes the type of
$S$). The fact that we get an odd integer as the type of an almost
symmetric duplication is not surprising, if we consider that
Rosales, in \cite{R}, proves that a numerical semigroup is one
half of a pseudo-symmetric numerical semigroup (i.e. an
almost-symmetric numerical semigroup of type two) if and only if
it is irreducible.

Finally, we notice that numerical duplication has an application
in the context of Weierstrass semigroups: using the fact that we
can control the genus of the numerical duplication and applying a
result of Torres in \cite{T}, the relation $S=(S\!\Join^b\!\!E)/2$
implies that, starting from a numerical semigroup that cannot be
realized as a Weierstrass semigroup, we can construct other
numerical semigroups that cannot be realized as Weierstrass
semigroups (see Remark \ref{Weierstrass}).

\medskip
The structure of the paper is the following: in the first section
we recall all the basic notions on numerical semigroups, that we
will use in the rest of the paper, and we prove some preliminary
lemmas about relative ideals. In Section 2, we define the
numerical duplication, we compute its Frobenius number, its genus
and we show that it is symmetric if and only if the ideal $E$ is a
canonical ideal (see Proposition \ref{Re}); moreover, in
Proposition \ref{type}, we compute the type of the duplication in
terms of $S$ and $E$.

In Section 3, we characterize when the numerical duplication is
almost symmetric (see Theorem \ref{main}) and, if this is the
case, we get new formulas to compute its type (see Proposition
\ref{type2}). Moreover, we show that, for any odd integer $x$
between $1$ and $2t(S)+1$, it is possible to obtain infinitely
many almost symmetric
numerical duplications with type $x$ (for every odd $b \in S$) and,
as a corollary, we get that every numerical semigroup $S$ is one
half of infinitely many almost symmetric semigroups with
prescribed odd type (not bigger than $2t(S)+1$).

\section{Preliminaries}

Most of the definitions and results that we recall in this section
can be found in \cite{FGH} or in \cite{RG}; if not, we will give
specific references. A {\it numerical semigroup} $S$ (briefly
n.s.) is a submonoid of $\mathbb N$, such that $|\mathbb N
\setminus S| < \infty$. The elements of $\mathbb N \setminus S$
are called gaps of $S$ and their cardinality is called the {\it
genus} of the semigroup, denoted by $g(S)$. The largest gap is
called {\it Frobenius number} of $S$ and it is denoted by $f(S)$.
The set $\{f(S)+1+x : \ x \in \mathbb N\}$ is called the {\it
conductor} of $S$ and it is denoted by $C(S)$ (the name conductor
is due to the equality $C(S)=\{z \in \mathbb Z: \ z+n \in S, \
\forall \ n \in \mathbb N\}$).

Let $M(S):=S\setminus\{0\}$ and $M(S)-M(S)=\{z \in \mathbb Z : \
z+s \in M(S), \ \forall \ s \in M(S) \}$; it is well known that
$M(S)-M(S)=S-M(S):=\{z \in \mathbb Z : \ z+s \in S, \ \forall \ s
\in M(S) \}$ and it is a n.s. containing both $S$ and $f(S)$. The
elements of $(M(S)-M(S))\setminus S$ are called {\it
pseudo-Frobenius numbers} of $S$, the cardinality
$|(M(S)-M(S))\setminus S|$ is said {\it type} of $S$ and it is
denoted by $t(S)$.

\medskip
It is clear that, if $s \in S$, then $f(S)-s \notin S$. Using this
remark, the gaps of $S$ are usually classified as gaps of the
first type, i.e. non-negative integers $x \notin S$ of the form
$x=f(S)-s$ with $s \in S$, and gaps of the second type, i.e.
non-negative integers $y \notin S$, such that $f(S)-y \notin S$.
By definition, all pseudo-Frobenius numbers except $f(S)$ are gaps
of the second type; if we denote by $L(S)$ the set of gaps of the
second type, we can express this property saying that $L(S)\cup
\{f(S)\} \supseteq (M(S)-M(S))\setminus S$.

A n.s. $S$ is said to be {\it symmetric} if and only if $\forall \
z \in \mathbb Z$,
$$
z \in S \ \Longleftrightarrow \ f(S)-z \notin S.
$$
It is well known that $S$ is symmetric if and only if $f(S)+1=2g(S)$;
another equivalent condition for $S$ to be symmetric is $t(S)=1$.
For symmetric semigroups all the non-negative integers not
belonging to $M(S)-M(S)$ are gaps of the first type; we can
rephrase this property saying that $L(S)\cup \{f(S)\} =
(M(S)-M(S))\setminus S$ (notice that, for symmetric semigroup
$L(S)=\emptyset$).

The last property mentioned above can be used (see \cite{BF}) to
define a new class of semigroups: $S$ is said to be {\it almost
symmetric} if $L(S)\cup \{f(S)\} = (M(S)-M(S))\setminus S$.
Symmetric semigroups are exactly the almost symmetric semigroups
of type one. Almost symmetric semigroups of type two are called
{\it pseudo-symmetric}.

\medskip

A set $E \subseteq \mathbb Z$ is said to be a {\it relative ideal}
of $S$ if $S+E\subseteq E$ (i.e. $s+t\in E$, for every $s \in S$
and $t \in E$) and there exists $s \in S$ such that $s+E=\{s+t: \
t \in E \} \subseteq S$. Relative ideals contained in $S$ are
simply called {\it ideals}. $S$, $M(S)$, $C(S)$ and $M(S)-M(S)$
are clearly relative ideals of $S$; in particular, $M(S)$ is
called {\it maximal ideal} of $S$. We can operate with relative
ideals obtaining again relative ideals: if $E$ and $F$ are
relative ideals of $S$, then $E+F=\{t+u: \ t \in E, \ u \in F\}$
and $E-F=\{z \in \mathbb Z: \ z+u \in E, \ \forall \ u \in F \}$
are both relative ideals of $S$. By $nE$ we will denote $E+E+\dots
+E$, $n$-times.

If $E$ is a relative ideal of $S$, we define $m(E)=\min E$, $f(E)=
\max (\mathbb Z \setminus E)$ (it is well defined since
$m(E)+C(S)\subseteq E$) and $g(E)=|(\mathbb Z \setminus E)\cap
\{m(E),m(E)+1, \dots, f(E)\}|$. One can always shift a relative
ideal $E$ adding to it an integer $x$: $x+E=\{x+e:\ e \in E\}$. It
is obvious that the relation $E \sim E' \Leftrightarrow E'=x+E$
(for some $x \in \mathbb Z$) is an equivalence relation. In any
equivalence class there is exactly one representative $\widetilde
E$ such that $f(\widetilde E)=f(S)$; it is obtained by any ideal
$E$ in the following way: $\widetilde E=E+f(S)-f(E)$. For all
ideals $E$ we have that $C(S) \subseteq \widetilde E \subseteq
\mathbb N$ (the second inclusion follows by the inclusion
$m(\widetilde E)+C(S) \subseteq \widetilde E$, that implies
$f(S)=f(\widetilde E) \leq m(\widetilde{E})+f(S)$, i.e. $0\leq
m(\widetilde E)$).

\medskip
There exists a distinguished class of relative ideals of $S$: the
class of canonical ideals. Following \cite{J}, we say that a
relative ideal $K'$ is a canonical ideal of $S$ if $K'-(K'-E)=E$,
for every relative ideal $E$ of $S$. In particular, $K(S)=\{x \in
\mathbb Z: \ f(S)-x \notin S \}$ is a canonical ideal (\cite[Satz
4]{J}) and we will call it {\it standard canonical ideal};
moreover, for any canonical ideal $K'$, we have $K'\sim K(S)$.
Notice that $f(K(S))=f(S)$.

It is straightforward to see that $S \subseteq K(S) \subseteq
\mathbb N$ and that $S$ is symmetric if and only if $K(S)=S$;
moreover, the following property holds.

\begin{lem}\label{1}(\cite[Hilfssatz 5]{J}) For any relative ideal $E$, $K(S)-E=\{x \in \mathbb
Z: \ f(S)-x \notin E\}$.
\end{lem}

Using the previous lemma other duality properties can be easily
proved:
\begin{itemize}
  \item $E \subseteq F  \Leftrightarrow K(S)-F \subseteq
  K(S)-E$ and $E=F \Leftrightarrow K(S)-F =
  K(S)-E$;
  \item if $E \subseteq F$, then $|F \setminus E|=|(K(S)-E)\setminus (K(S)-F)|$;
  \item $K(S)-K(S)=S$.
\end{itemize}
Moreover, by \cite[Proposition 12]{RB} it follows that $K(S)$ is
generated, as relative ideal of $S$, by $t(S)$ elements and these
generators are exactly of the form $f(S)-x$, where $x \in
(M(S)-M(S)) \setminus S$.

The standard canonical ideal $K(S)$ can be used to give an
equivalent condition for $S$ to be almost symmetric (see
\cite[Proposition 4]{BF}):
$$S \ \text{almost symmetric} \ \Leftrightarrow \ K(S)+M(S)\subseteq M(S);$$
we will use this characterization in the third section.

\medskip
In this paper we will need to consider, in particular, the
relative ideal \ $K(S)-(M(S)-M(S))=\{x \in \mathbb Z: \ f(S)-x
\notin M(S)-M(S)\}$ (cf. Lemma \ref{1}). \ Since  \ $S
\subseteq M(S)-M(S) \subseteq \mathbb N$, \ by duality properties
we obtain \ $C(S)=K(S)-\mathbb N \subseteq K(S)-(M(S)-M(S))
\subseteq K(S)$.

Notice also that \ $K(S)+M(S) \subseteq K(S)-(M(S)-M(S))$; \ in
fact, if $x =y+s$ (with $y\in K(S)$ and $s\in M(S)$) and $z \in
M(S)-M(S)$, we get $x+z=y+(s+z) \in K(S)+M(S) \subseteq K(S)$.

%\begin{proof}
%Let $x$ be an integer such that $f(S)-x \notin M-M$; then there
%exists $m \in M$ such that $f(S)-x+m \notin M$. Hence either
%$f(S)-x+m=0$ or $f(S)-x+m \notin S$. In the first case, we get
%$x=f(S)+m>f(S)$ and so $x \in K-(M-M)$. The second condition is
%equivalent to $x-m \in K$, hence there exists $k \in K$ such that
%$x=m+k$. Now pick an element $y$ in $M-M$: $x+y=(y+m)+k \in M+K
%\subset K$; hence $x+(M-M) \subseteq K$.
%
%Conversely, let $x \in K-(M-M)$. If $f(S)-x \in M-M$ then
%$f(S)=x+(f(S)-x) \in K$, a contradiction.
%\end{proof}

\medskip
We will need also to use the following facts.

\begin{lem}\label{K-biggest} Let $E$ be a relative ideal of $S$; then $\widetilde
E=E+f(S)-f(E)$ is such that $C(S) \subseteq \widetilde E \subseteq
K(S)$.
\end{lem}

\begin{proof}
The first inclusion is obvious. The second inclusion follows by
definitions of relative ideal and of $K(S)$: if $x \in \widetilde
E \setminus K(S)$, then $f(S)-x \in S$, that implies
$(f(S)-x)+x=f(S) \in \widetilde E$; contradiction.
\end{proof}

\begin{lem}\label{cardinality} Let $E$ be a relative  ideal of a n.s. $S$. Then
$f(S)+1-g(S)\leq g(\widetilde E)+m(\widetilde E)$ and equality
holds if and only if $E$ is a canonical ideal (i.e. $\widetilde
E=K(S)$).
\end{lem}

\begin{proof}
The integer $g(\widetilde E)+m(\widetilde E)$ is the number of
elements in $\mathbb N \setminus \widetilde E$. On the other hand
$f(S)+1-g(S)$ is the number of elements $s$ in $S$, smaller than
$f(S)+1$. Since $s \in S$ implies $f(S)-s \notin K(S) \supseteq
\widetilde E$, the thesis follows by the definition of $K(S)$.
\end{proof}

\section{The numerical duplication of a numerical semigroup}

In this section we define the numerical duplication and we study
some of its basic properties. We fix the notations that we will
use in the rest of the paper: with $S$ we denote a n.s., with $M$,
$C$ and $K$ its maximal ideal, its conductor and its standard
canonical ideal, respectively; we set $f=f(S)$ and $g=g(S)$. Let
$E\subseteq S$ be an ideal of $S$; we define $e=f(E)-f$ and
$\widetilde E=E-e$ (so $f(\widetilde E)=f(K)=f$).

We also set $2\cdot S:=\{2s: \ s \in S \}$ and $2\cdot E:=\{2t: \
t \in E \}$ (notice $2\cdot S\neq 2S=S+S$ and $2\cdot E \neq
2E=E+E$). Let $b \in S$ be an odd integer. Then we define the {\it
numerical duplication}, $S\!\Join^b\!\!E$, of $S$ with respect to
$E$ and $b$ as the following subset of $\mathbb N$:
$$
S\!\Join^b\!\!E=2\cdot S \cup (2\cdot E+b).
$$
It is straightforward to check that $S\!\Join^b\!\!E$ is a
numerical semigroup. In fact $0=2\cdot 0 \in S\!\Join^b\!\!E$;
since $f(E) \geq f$, every integer $n>2f(E)+b$ belongs to
$S\!\Join^b\!\!E$; finally, the conditions $b\in S$ and $E$ ideal
of $S$ immediately imply that $S\!\Join^b\!\! E$ is closed with
respect to the sum.

\medskip Notice that, more generally, the previous construction
produces a n.s. only assuming that $E$ is a relative ideal of $S$,
that $b \in S$ is an odd integer and that $b+E+E \subseteq S$ (the
last condition is fulfilled, e.g., if $b$ is big enough), but for
our aims it is simpler to use the hypotheses assumed at the
beginning of this section.

\begin{prop}\label{Re} The following properties hold for $S\!\Join^b\!\!E$:
\begin{description}
  \item[(1)] $f(S\!\Join^b\!\!E)=2f(E)+b$;
  \item[(2)] $g(S\!\Join^b\!\!E)= g+g(E)+m(E)+\frac{b-1}{2}$;
  \item[(3)] $S\!\Join^b\!\!E$ is symmetric if and only if $E$ is a
  canonical ideal of $S$.
\end{description}
\end{prop}

\begin{proof}
As we observed above, property (1) follows by $f(E) \geq f$. As
for property (2), we first note that even gaps of
$S\!\Join^b\!\!E$ correspond bijectively to the gaps of $S$;
secondly, every odd integer smaller than $2m(E)+b$ is not in
$S\!\Join^b\!\!E$; finally, if $x \notin S\!\Join^b\!\!E$ is an
odd integer such that $2m(E)+b \leq x \leq 2f(E)+b$, then
$x=2y+b$, with $y \notin E$ and $m(E)\leq y \leq f(E)$.

As for point (3), $S\!\Join^b\!\!E$ is symmetric if and only if
$f(S\!\Join^b\!\!E)+1=2(g(S\!\Join^b\!\!E))$. Using (1) and (2) we
get that $S\!\Join^b\!\!E$ is symmetric if and only if
$$\aligned
2&f(E)+b+1=2g+2g(E)+2m(E)+b-1 \\
&\Longleftrightarrow\ 2(f(E)+1)=2(g+g(E)+m(E)). \endaligned
$$
With the notation fixed at the beginning of this section, we have
$f(E)=f+e$, $m(E)=m(\widetilde E)+e$ and $g(E)=g(\widetilde E)$; therefore the last
equality divided by $2$, i.e. $f(E)+1=g+g(E)+m(E)$ is equivalent
to $f+1=g+g(\widetilde E)+m(\widetilde E)$ or, equivalently, to
$f+1-g=g(\widetilde E)+m(\widetilde E)$. Hence, by Lemma
\ref{cardinality}, $S$ is symmetric if and only if $\widetilde
E=K$, that is $E$ is a canonical ideal of $S$.
\end{proof}

\begin{rem}
Point (3) of the previous proposition was observed yet in
\cite{DR} (and we publish it with the permission of the second
author) and gives, in particular, an alternative proof of the
results of \cite{RGS} and \cite{RGS2}, where it is stated that
every n.s. is one half of a (respectively, infinitely many)
symmetric n.s.; moreover, the construction given in \cite{RGS2}
essentially coincides with this construction, since $K$ is
generated, as relative ideal, by the elements $f(S)-x$, where $x$
varies in the set of the pseudo-Frobenius numbers of $S$.
\end{rem}

\begin{rem}\label{Weierstrass}
In \cite{T} has been defined the notion of $\gamma$-{\it
hyperelliptic} n.s., that is a n.s. with $\gamma$ even gaps, in
the context of Weierstrass semigroups theory; in particular, this
notion and its generalizations have been used to find semigroups
that cannot be realized as Weierstrass semigroups (see also
\cite{T2}). It is clear that the numerical duplication produces
$g$-hyperelliptic semigroups. If we try to apply directly the
criterion given in \cite[Remark 3.4]{T}, we cannot decide if a
numerical duplication is Weierstrass or not. On the other hand,
arguing as in \cite[Corollary 3.3]{T}, if we start with a n.s. $S$
which cannot be realized as a Weierstrass semigroup and if we
choose $E$ and $b$ such that $g(S\!\Join^b\!\!E)>6g+4$, we can
conclude that also $S\!\Join^b\!\!E$ cannot be realized as a
Weierstrass semigroup.
\end{rem}

\begin{rem} Let $n \in \mathbb N$, $n \geq 2$ and let $b\in S$,
relatively prime to $n$. Then, setting $n\cdot E=\{nt:\ t \in
E\}$, we can generalize the numerical duplication as follows:
$$S\!\Join^b_n\!\!E=n\cdot S \cup (n\cdot
E+b) \cup (n\cdot(2E)+2b) \cup \dots \cup (n\cdot((n-1)E)+(n-1)b).
$$
We have that $S\!\Join^b_n\!\!E$ is a n.s., with the property that
$S$ is one over $n$ of it. We call $S\!\Join^b_n\!\!E$ the {\it
numerical $n$-tuplication} of $S$ with respect to $E$ and $b$. The
fact that $S\!\Join^b_n\!\!E$ is a n.s. follows by a case by case
computation, using the property $hE+kE=(h+k)E$, for any $h,k \in
\mathbb N$ (where $0E=S$, by definition).
\end{rem}

In the next section we will need to consider the standard
canonical ideal of the numerical duplication: $K(S\!\Join^b\!\!E)=
\{z \in \mathbb Z : \ 2f(E)+b-z \notin S\!\Join^b\!\!E \}$. We
have that $a=2f(E)+b-z \notin S\!\Join^b\!\!E$ if and only if
either $a$ is even and $\frac{a}{2} \notin S$ or $a$ is odd and
$\frac{a-b}{2} \notin E$. Hence we obtain:
$$
z \in K(S\!\Join^b\!\!E) \ \Longleftrightarrow \ z=2f(E)+b-a \ \
\text{with} \ \
  \begin{cases}
    \frac{a}{2} \notin S,  & a \ \text{even}, \\
    \frac{a-b}{2} \notin E, & a \ \text{odd}.
  \end{cases}
$$

We now study the type of the numerical duplication $S\!\Join^b\!\!
E$.

\begin{prop}\label{type}
Let $S$ be a n.s., let $M$ be its maximal ideal, let $E\subseteq
S$ be an ideal of $S$ and let $b\in S$ be an odd integer. Then the
type of $S\!\Join^b\!\! E$ is
$$
t(S\!\Join^b\!\! E)=|((M-M)\cap (E-E))\setminus S|+|(E-M)\setminus
E|
$$
In particular, $t(S\!\Join^b\!\! E)$ does not depend on $b$.
\end{prop}

\begin{proof} Let $T=S\!\Join^b\!\!E$.
Let $x=2h$ be an even integer not belonging to $T$, i.e. $h \notin
S$; $x \in M(T)-M(T)$ if and only if $2h+2s \in M(T)$, for every
$s \in M$, and $2h+2t+b \in M(T)$, for every $t \in E$. These two
conditions are equivalent to $h \in M-M$ and $h\in E-E$,
respectively; hence we get the first summand of the formula in the
statement.

Let now $x=2h+b$ be an odd integer not belonging to $T$, i.e. $h
\notin E$; $x \in M(T)-M(T)$ if and only if $2h+b+2s \in M(T)$,
for every $s \in M$, and $2h+b+2t+b \in M(T)$, for every $t \in
E$. In this case the two conditions are equivalent to $h \in E-M$
and $h \in M-(b+E)$, respectively. Assume that $E \subseteq M$;
since $b \in M$, the condition $h \in E-M$ implies $h+b+E
\subseteq E+E$; but $E \subseteq M$, so $E+E \subseteq M$ and thus
$h \in M-(b+E)$. Hence we get $E-M \subseteq M-(b+E)$, that
implies the second summand of the formula in the statement. It
remains to prove the thesis in the case $0\in E$, that is $E=S$: the first condition
becomes  $h \in S-M=M-M$; hence $h+b+S \subseteq M+S=M$ and again
we have $S-M \subseteq M-(b+S)$, that implies the thesis.
\end{proof}

\begin{ex}\label{extype}
Let $S=\{0,5, \rightarrow\}$ and $E=\{5,8,10, \rightarrow\}$; let
$b=5$. The numerical duplication is
$S\!\Join^b\!\!E=\{0,10,12,14,15,16,18,20,21,22,24,\rightarrow\}$,
which has Frobenius number equal to $23=18+5=2f(E)+b$ and genus
equal to $14=4+3+5+2=g+g(E)+m(E)+\frac{b-1}{2}$. As for the type,
we have $(M-M)\cap (E-E)=\mathbb N \cap \{0,3,5,\rightarrow \}$
and $E-M=\{5,\rightarrow\}$. Hence $t(S\!\Join^b\!\!
E)=|[(M-M)\cap (E-E)]\setminus S|+|(E-M)\setminus E|=1+3=4$, as
can be checked directly, computing the pseudo-Frobenius numbers of
the numerical duplication: $\{6,17,19,23\}$.
\end{ex}

\section{Almost symmetric duplications}

We start this section giving the proof of the main result of the
paper. We recall that we use the notations introduced at the
beginning of Section 2. We will need some lemmas.

%\begin{lem}\label{1}
%Let $E$ be an ideal of a n.s. $S$ and let $\widetilde E=E-e$.
%Assume that $K-(M-M) \subseteq \widetilde E$.
%Then, for any $y\in (M-M)$, we have $y \notin K-\widetilde E$ \ if
%and only if \ $f-y \in \widetilde E$.
%\end{lem}
%
%\begin{proof}
%Assume that $f-y \in \widetilde E$; if $y \in K-\widetilde E$,
%then $f=(f-y)+y \in K$, a contradiction.
%
%Conversely, assume that $y \notin K-\widetilde E$. Hence there
%exists $z \in \widetilde E$, such that $y+z \notin K$ or,
%equivalently, $f-y-z \in S$. If  $f-y-z \neq 0$, i.e. it belongs
%to $M$, then $y \in M-M$ implies that $f-z \in M$ and so
%$f=(f-z)+z \in \widetilde E$; contradiction. Hence $f-y-z=0$ and,
%therefore, $f-y=z \in \widetilde E$.
%\end{proof}

\begin{lem}\label{2}
Let $E$ be an ideal of a n.s. $S$ and let $\widetilde E=E-e$.
Assume that $K-(M-M) \subseteq \widetilde E$. Then, for any $x
\notin E$, $f(E)-x \in M-M$.
\end{lem}

\begin{proof}
Since $x \notin E$, we have $x-e \notin \widetilde E \supseteq
K-(M-M)$. Hence there exists $y \in M-M$, such that $x-e+y \notin
K$, i.e. $f+e-x-y =f(E)-x-y \in S$. Since $y \in M-M$, which is a
relative ideal of $S$, $f(E)-x=(f(E)-x-y)+y \in S+(M-M)=M-M$.
\end{proof}

\begin{lem}\label{3}
Let $E$ be an ideal of a n.s. $S$ and let $\widetilde E=E-e$.
Assume that $K-(M-M) \subseteq \widetilde E$ and that
$K-\widetilde E$ is a numerical semigroup. Then, for any $x \notin
E$, $f(E)-x \in E-E$.
\end{lem}

\begin{proof}
As in the proof of the previous lemma, there exists $y \in M-M$,
such that $f(E)-x-y \in S$. If $f(E)-x-y \in M$, then $f(E)-x \in
M$ and, therefore, $f(E)-x+t \in E$, for every $t \in E$ (since
$E$ is an ideal of $S$).

It remains to prove the thesis in the case $f(E)-x-y =0$. Since
$f(E)=f+e$, we get $f-y=x-e \notin \widetilde E$; hence, applying
Lemma \ref{1}, it follows that $y \in K-\widetilde E$. We need to
show that, for every $t \in E$, $y+t \in E$, i.e. $y+t-e \in
\widetilde E$. Assume, by contradiction, that $y+t-e \notin
\widetilde E$; applying again Lemma \ref{1}, we obtain $f-(y+t-e)
\in K-\widetilde E$; also $y \in K-\widetilde E$, which is a n.s.,
hence $f-t+e=[f-(y+t-e)]+y\in K-\widetilde E$ and, since $t-e \in
\widetilde E$, we get $f \in K$; contradiction.
\end{proof}

%Let $t \in E$; we need to show that $f(E)-x+t \in E$. Set
%$e=f(E)-f(S)$; as in the proof of the previous lemma, we have that
%there exists $y \in M-M$, such that $f(E)-x-y \in S$. If $f(E)-x-y
%\in M$, then $f(E)-x \in M$ and, therefore, $f(E)-x+t \in E$
%(since $E$ is an ideal of $S$). It remains to prove the thesis in
%the case $f(E)-x-y =0$. Since $f(E)=f+e$, we get $f-y=x-e \notin
%\widetilde E$; moreover, $y \in M-M$, hence, applying Lemma
%\ref{1}, it follows that $y \in K-\widetilde E$. We need to show
%that $y+t \in E$, i.e. $y+t-e \in \widetilde E$.
%
%We have now two possible cases: either $f-(y+t-e) \in M-M$ or
%$f-(y+t-e) \notin M-M$. In the first case, assume by contradiction
%that $y+t-e \notin \widetilde E$; applying again Lemma \ref{1}, we
%obtain $f-(y+t-e) \in K-\widetilde E$; also $y \in K-\widetilde
%E$, which is a n.s., hence $f-t+e=[f-(y+t-e)]+y\in K-\widetilde E$
%and, since $t-e \in \widetilde E$, we get $f \in K$;
%contradiction.
%
%In the second case, that is $f-(y+t-e) \notin M-M=S-M$, there
%exists $s \in M$ such that $f-(y+t-e)+s \notin S$, i.e. $y+t-e-s
%\in K$; it follows that $y+t \in e+K+M \subseteq e+(K-(M-M))
%\subseteq e+\widetilde E=E$.
%\end{proof}

\begin{thm}\label{main}
Let $E$ be an ideal of a n.s. $S$; let $\widetilde E=E-e$. Then
$S\!\Join^b\!\! E$ is almost symmetric if and only if $K-(M-M)
\subseteq \widetilde E \subseteq K$ and $K-\widetilde E$ is a
numerical semigroup.
\end{thm}

\begin{proof}
Let $T=S\!\Join^b\!\! E$, $M(T)$ its maximal ideal and $K(T)$ its
standard canonical ideal. Assume that $T$ is almost symmetric,
that is $K(T)+M(T) \subseteq M(T)$. By Lemma \ref{K-biggest} we
know that $\widetilde E \subseteq K$. Pick now $y \in K-(M-M)$ and
assume that $y \notin \widetilde E$ (equivalently, $y+e \notin
E$). We have $f-y=f(E)-(y+e)$; hence, remembering the description
of $K(T)$ given in Section 2, we obtain $2(f-y)=2f(E)+b-(2(y+e)+b)
\in K(T)$, since $2(y+e)+b$ is odd and $y+e \notin E$. Since $T$
is almost symmetric, it follows that $2(f-y)+2s \in M(T)$ for
every $s \in M$, that is $f-y+s \in M$ for every $s \in M$, i.e.
$f-y \in M-M$; but $y \in K-(M-M)$ and therefore $f
 \in K$; contradiction. Hence
$K-(M-M) \subseteq \widetilde E$.

It remains to show that $K-\widetilde E$ is a numerical semigroup.
Since $\widetilde E \subseteq K$ it is clear that $0 \in
K-\widetilde E$. By $K-(M-M) \subseteq \widetilde E \subseteq K$
and duality properties, we also have that $S \subseteq
K-\widetilde E \subseteq M-M$; in particular, $|\mathbb N\setminus
(K-\widetilde E)| < \infty$.

Let $y$ and $z$ be two elements of $K-\widetilde E$.
%Since $K-\widetilde E \subseteq M-M$ which is a n.s., $y+z \in M-M$.
Assume that $y+z \notin K-\widetilde E$;
%there exists $w \in
%\widetilde E$ such that $y+z+w \notin K$, i.e. $f-(y+z)-w=s \in
%S$. If $s \neq 0$, then $(y+z)+s \in M$ and so $f=[(y+z)+s]+w \in
%M+\widetilde E\subseteq \widetilde E$, a contradiction; hence
%$s=0$ and
hence, by Lemma \ref{1}, $f-(y+z) \in \widetilde E$ or,
equivalently, $f(E)-(y+z) \in E$. It follows that $2f(E)-2(y+z)+b
\in M(T)$. Moreover, since $y \in K-\widetilde E$, applying again
Lemma \ref{1}, we get $f-y \notin \widetilde E$, that is $f(E)-y
\notin E$; thus $2f(E)+b-[2(f(E)-y)+b] =2y \in K(T)$; analogously,
$2z \in K(T)$. Now, by $K(T)+M(T) \subseteq M(T)$, it follows that
$f(T)=2f(E)+b=2y+[2z+(2f(E)-2(y+z)+b)] \in M(T)$; contradiction.
Therefore $y+z \in K-\widetilde E$ and it is a n.s.

\smallskip
Conversely, assume that $K-(M-M) \subseteq \widetilde E \subseteq
K$ and $K-\widetilde E$ is a numerical semigroup; we need to show
that $K(T)+M(T) \subseteq M(T)$. Using the descriptions of the
elements of $T$ and $K(T)$, we have to consider four cases:
\begin{description}
  \item[(i)] $2s \in M(T)$ and $2f(E)+b-a \in K(T)$, where $s \in
  M$, $a$ is even and $\frac{a}{2} \notin S$;
  \item[(ii)] $2s \in M(T)$ and $2f(E)+b-a \in K(T)$, where $s \in
  M$, $a$ is odd and $\frac{a-b}{2} \notin E$;
  \item[(iii)] $2t+b \in M(T)$ and $2f(E)+b-a \in K(T)$, where $t \in
  E$, $a$ is even and $\frac{a}{2} \notin S$;
  \item[(iv)] $2t+b \in M(T)$ and $2f(E)+b-a \in K(T)$, where $t \in
  E$, $a$ is odd and $\frac{a-b}{2} \notin E$.
\end{description}
{\bf (i)} Since \ $2s+2f(E)+b-a$ \ is odd, it belongs to $M(T)$ if
and only if \ $s+f(E)-\frac{a}{2} \in E$, \ i.e.  $s+f-\frac{a}{2}
\in \widetilde E$. Since $\frac{a}{2} \notin S$, i.e.
$f-\frac{a}{2} \in K$, we obtain $s+f-\frac{a}{2}\in M+K \subseteq
K-(M-M) \subseteq \widetilde E$.

\noindent {\bf (ii)} Since $2s+2f(E)+b-a$ is even, it belongs to
$M(T)$ if and only if $s+f(E)-\frac{a-b}{2} \in M$. Since
$\frac{a-b}{2}\notin E$, we can apply Lemma \ref{2} to obtain
$f(E)-\frac{a-b}{2} \in M-M$, that implies the thesis.

\noindent {\bf (iii)} Since $2t+b+2f(E)+b-a$ is even, it belongs
to $M(T)$ if and only if $t+b+f(E)-\frac{a}{2} \in M$. Since $t+b
\in E \setminus \{0\} \subseteq M$ and $\frac{a}{2} \notin
S\supseteq E$, the thesis follows by Lemma \ref{2}.

\noindent {\bf (iv)} Since $2t+b+2f(E)+b-a$ is odd, it belongs to
$M(T)$ if and only if $t+f(E)-\frac{a-b}{2} \in E$. Since
$\frac{a-b}{2} \notin E$, the thesis follows immediately by Lemma
\ref{3}. \end{proof}

\begin{ex}
Let $S=\{0,4,8 \rightarrow \}$; we have $K=\{0,1,2,4,5,6,8
\rightarrow \}$ and $K-(M-M)=\{4,5,6,8 \rightarrow \}$. All
possible relative ideals $\widetilde{E}$ between $K-(M-M)$ and $K$
and the corresponding relative ideals $K-\widetilde{E}$ are listed
below:
\[
\begin{array}{ll}
\widetilde{E}_1=K-(M-M)=\{4,5,6,8 \rightarrow\} \ \ \ \ \ \ \ \ & K-\widetilde{E}_1=M-M=\{0,4 \rightarrow\} \\
\widetilde{E}_2=\{2,4,5,6,8 \rightarrow\} \ \ \ \ \ \ \ \ \ \ \ \ & K-\widetilde{E}_2=\{0,4,6,7,8 \rightarrow\} \\
\widetilde{E}_3=\{1,4,5,6,8 \rightarrow\} \ \ \ \ \ \ \ \ \ \ \ \ & K-\widetilde{E}_3=\{0,4,5,7,8 \rightarrow\} \\
\widetilde{E}_4=\{0,4,5,6,8 \rightarrow\} \ \ \ \ \ \ \ \ \ \ \ \ & K-\widetilde{E}_4=\{0,4,5,6,8 \rightarrow\} \\
\widetilde{E}_5=\{1,2,4,5,6,8 \rightarrow\} \ \ \ \ \ \ \ \ \ \ \ \ & K-\widetilde{E}_5=\{0,4,7,8 \rightarrow\} \\
\widetilde{E}_6=\{0,2,4,5,6,8 \rightarrow\} \ \ \ \ \ \ \ \ \ \ \ \ & K-\widetilde{E}_6=\{0,4,6,8 \rightarrow\} \\
\widetilde{E}_7=\{0,1,4,5,6,8 \rightarrow\} \ \ \ \ \ \ \ \ \ \ \ \ & K-\widetilde{E}_7=\{0,4,5,8 \rightarrow\} \\
\widetilde{E}_8=K=\{0,1,2,4,5,6,8 \rightarrow\} \ \ \ \ \ \ \ \ \
\ \ \ & K-\widetilde{E}_8=S=\{0,4,8 \rightarrow\}.
\end{array}
\]
It is straightforward to check that $K-\widetilde{E}_i$ is a n.s.
for all $i=1, \dots,8$, so $S\!\Join^b\!\!E$ (with $E\subseteq S$
ideal of $S$) is almost symmetric if and only if $E \sim
\widetilde{E}_i$ for some $i=1, \dots,8$, i.e. there exists $x \in
\mathbb{N}$ such that $E=x+\widetilde{E}_i$.
\end{ex}

\begin{ex} \label{example}
In the previous example any ideal $\widetilde{E}$ between
$K-(M-M)$ and $K$ is such that $K-\widetilde{E}$ is a n.s.; in
general, this is not the case; for example, if we consider the
n.s. $S=\{0,5, \rightarrow\}$ of Example \ref{extype}, we have
$K=\{0,1,2,3,5 \rightarrow\}$ and $K-(M-M)=\{5, \rightarrow\}$; if
we consider the ideal $E=\{5,8,10\}$, we get
$\widetilde{E}=\{0,3,5 \rightarrow\}$ and
$K-\widetilde{E}=\{0,2,3,5 \rightarrow\}$, that is not a n.s.. \\
Thanks to Theorem \ref{main},
$S\!\Join^b\!\!E=\{0,10,12,14,15,16,18,20,21,22,24, \rightarrow
\}$ is not almost symmetric; in fact $4 \in K(S\!\Join^b\!\!E)$
and $15 \in M(S\!\Join^b\!\!E)$, but $19=4+15 \notin
M(S\!\Join^b\!\!E)$.
\end{ex}

It is possible to characterize those semigroups such that for any
ideal $\widetilde{E}$ between $K-(M-M)$ and $K$, $K-\widetilde{E}$
is a n.s., using the so called {\it Ap\'ery set} of $S$ with
respect to a nonzero element $n \in S$: $\ap(S,n)=\{w_0,w_1,
\dots,w_{n-1}\}$, where $w_i=\min\{s \in S: \ s \equiv i \
(\!\!\!\!\mod \ n)\}$. Consider the following partial ordering
$\leq_S$ on $\ap(S,n)$: $w_i \leq_S w_j$ if $w_j - w_i \in S$. It
is well known that $w_i \in \ap(S,n)$ is maximal with respect to
$\leq_S$ if and only if $w_i-n$ is a pseudo-Frobenius number of
$S$ (see, e.g., \cite[Proposition 2.20]{RG}).

\begin{prop}
Let $S$ be a numerical semigroup, $n$ a nonzero element of $S$ and
$\ap(S,n)=\{w_0,w_1, \dots,w_{n-1} \}$. The following conditions
are equivalent:
\begin{description}
  \item[(i)] For every relative ideal $\widetilde E$,
  such that $K-(M-M)\subseteq  \widetilde{E} \subseteq K$,
  $K-\widetilde{E}$ is a n.s..
  \item[(ii)] For every $w_i,w_j,w_k \in \ap(S,n)$, maximal with respect to
  $\leq_S$, $w_i+w_j\neq w_k + n$.
\end{description}
\end{prop}

\begin{proof}
Remembering that $S \subseteq K-\widetilde{E} \subseteq M-M$, we
have that $K-\widetilde{E}$ is a n.s. if and only if it is closed
with respect to the sum.

Let $x,y$ be two nonzero elements of $K-\widetilde{E}$; if one of
them belongs $S$, the sum is still in $S$ because the other one is
in $M-M$. Hence condition (i) is satisfied if and only if, for any
choice of $\widetilde{E}$, for any $x,y \in (K-\widetilde{E})
\setminus S$ their sum belongs to $K-\widetilde{E}$. This
condition implies that, for every two elements $x,y$ of $(M-M)
\setminus S$, i.e. two pseudo-Frobenius numbers, their sum belongs
to $S$; otherwise, if we set $F=S \cup \{x,y\}$ and
$\widetilde{E}=K-F$, we have $K-(M-M) \subseteq \widetilde{E}
\subseteq K$, but $K-\widetilde{E}=F$ is not a numerical
semigroup. Conversely, if the sum of any two pseudo-Frobenius
numbers belongs to $S$, it is clear that $K-\widetilde{E}
\subseteq M-M$ is a n.s. for every $\widetilde E$.

Finally, using the characterization of maximal elements in the
Ap\'ery set recalled above, we obtain the equivalence with
condition (ii).
\end{proof}

\medskip
In case $S\!\Join^b\!\! E$ is almost symmetric it is possible to
specialize Proposition \ref{type}, obtaining better formulas for
its type. Let us start with lemma.

\begin{lem} Let $E$ be an ideal of a n.s. $S$ and let $\widetilde E=E-e$.
Assume that $K-(M-M) \subseteq \widetilde E$ (i.e. $K-\widetilde E
\subseteq M-M$). Then the map $\varphi: y \mapsto f-y$ induces a
bijection between $(K-\widetilde E)\setminus S$ and $(\widetilde
E-M)\setminus (\widetilde E \cup \{f\})$.

In particular, $|(K - \widetilde E)\setminus S| =|(\widetilde
E-M)\setminus \widetilde E|-1$.
\end{lem}

\begin{proof}
We first prove that $\varphi:(K-\widetilde E)\setminus S
\longrightarrow (\widetilde E-M)\setminus  (\widetilde E \cup
\{f\})$ is well defined. Let $y \in (K-\widetilde E)\setminus S$;
by Lemma \ref{1}, $f-y \notin \widetilde E$; moreover, $y \notin
S$, hence $y \neq 0$ and so $f-y \neq f$. It remains to show that
$f-y \in \widetilde E-M$. Assume that there exists $s \in M$ such
that $f-y+s \notin \widetilde E=K-(K-\widetilde E)$; thus there
exists $z\in K-\widetilde E$, such that $f-y+s+z\notin K$, i.e.
$y-s-z \in S$; since $z\in K-\widetilde E\subseteq M-M$, we get
$s+z \in M$ and therefore $y=(y-s-z)+(s+z) \in S+M\subseteq S$;
contradiction against the choice of $y$.

Since $\varphi$ is clearly injective, we need only to show that it
is surjective: for any $z \in (\widetilde E-M)\setminus
(\widetilde E \cup \{f\})$, $f-z \in (K-\widetilde E)\setminus
S$. Assume that $f-z \in S$; since $z \neq f$, $f-z \in M$; hence
$f=z+(f-z) \in (\widetilde E-M)+M\subseteq \widetilde E$, a
contradiction. Finally, since $z \notin \widetilde E$, by Lemma
\ref{1} $f-z \in K-\widetilde E$.
\end{proof}

\begin{prop}\label{type2}
Let $E$ be an ideal of a n.s. $S$ and let $\widetilde E=E-e$.
Assume that $S\!\Join^b\!\! E$ is almost symmetric. Then the type
of the numerical duplication is
$$t(S\!\Join^b\!\! E) = 2|(E-M)\setminus
E|-1=2|(K - \widetilde E)\setminus S|+1=2|K \setminus \widetilde{E}|+1
$$
In particular, $t(S\!\Join^b\!\! E)$ is an odd integer, $1 \leq
t(S\!\Join^b\!\! E)\leq 2t(S)+1$ and it does not depend on $b$.
\end{prop}

\begin{proof} By the main theorem, we have
$K-(M-M) \subseteq \widetilde E \subseteq K$ and $K-\widetilde E$
is a numerical semigroup. Hence, $E-E=\widetilde E-\widetilde
E\subseteq K-\widetilde E \subseteq M-M$, so the formula in
Proposition \ref{type} becomes $$t(S\!\Join^b\!\! E)=|(\widetilde
E-\widetilde E)\setminus S|+|(E-M)\setminus E|.
$$
We claim that under our hypotheses, $K-\widetilde E= \widetilde
E-\widetilde E$. Assume that there exists $y \in (K-\widetilde E)
\setminus (\widetilde E-\widetilde E)$; hence there exists $t-e
\in \widetilde E$, with $t \in E$, such that $y+t-e \in K \setminus \widetilde E$;
%in particular, $f-y-t+e \notin S$. We have two cases: either
%$f-y-t+e \in M-M$ or not. If $f-y-t+e \in M-M$,
hence, by Lemma \ref{1}, $f-y-t+e \in K - \widetilde E$. Since $y
\in K-\widetilde E$ which is a n.s., we get $f-t+e \in K-
\widetilde E$, that implies $f=f-t+e+(t-e) \in K$; contradiction.
%As for the second case, i.e. $f-y-t+e \notin
%M-M=S-M$, we have that there exists $s \in M$ such that
%$f(S)-y-t+e+s \notin S$, that is $y+t-e-s \in K$; it follows that
%$y+(t-e) \in M+K \subseteq K-(M-M) \subseteq \widetilde E$, a
%contradiction against the assumption.

Hence the claim is proved and $$t(S\!\Join^b\!\! E)=|(K-\widetilde
E)\setminus S|+|(E-M)\setminus E|.
$$
Since $(E-M)\setminus E=[(\widetilde E-M)\setminus \widetilde
E]+e$, the first two equalities of the formula in the statement follow immediately by the
previous lemma. The last equality is a direct consequence of duality properties.

As for the second part of the statement, we need
only to show the inequality $t(S\!\Join^b\!\! E)\leq 2t(S)+1$. We
have that $t(S)=|(M-M)\setminus S|$; since $S \subseteq
K-\widetilde E \subseteq M-M$, we have $|(K-\widetilde E)\setminus
S|\leq t(S)$, as desired.
\end{proof}

Notice that, if $S\!\Join^b\!\!E$ is not almost symmetric, its
type may be even. For instance the numerical duplication presented
in Examples \ref{extype} and \ref{example} has type $4$.

\medskip
Notice that it is possible to obtain any odd integer $x=2m+1$ in
the range prescribed by the previous proposition. In fact, if
$\widetilde E=K$, then $|(K-\widetilde E)\setminus S|=0$ and
$t(S\!\Join^b\!\! E)=1$, as we knew by Proposition \ref{Re}. On
the other hand, if $\widetilde E=K-(M-M)$, then $K-\widetilde
E=M-M$ and $t(S\!\Join^b\!\! E)=2t(S)+1$. So the bounds obtained
in the proposition are sharp. To see that any odd integer in the
prescribed range can be achieved, we can argue as follows: let $F$
be the relative ideal obtained adding to $S$ the $m$ biggest
elements of $(M-M)\setminus S$ and consider $\widetilde{E}=K-F$.
It is straightforward to check that $F$ is a n.s., so if we
consider the ideal $E=\widetilde{E}+z$, for $z\in \mathbb N$ such
that $E\subseteq S$, we have that $S\!\Join^b\!\!E$ is almost
symmetric (since $K-\widetilde{E}=K-(K-F)=F$ is a n.s.); moreover
we have $t(S\!\Join^b\!\!E)=2|(K-\widetilde{E})\setminus
S|+1=2|F\setminus S|+1=2m+1$. Hence we have proved the following
corollary:

\begin{cor}
Let $S$ be a n.s. and let $x$ be any odd integer, $1\leq x\leq
2t(S)+1$; then, for every odd $b\in S$, there exist infinitely many ideals $E\subseteq S$ such that
$S\!\Join^b\!\! E$ is almost symmetric and $t(S\!\Join^b\!\!
E)=x$.
\end{cor}

As a byproduct of the previous results we obtain the following

\begin{cor}
Let $S$ be a n.s. and let $x$ be any odd integer, $1\leq x\leq
2t(S)+1$; then $S$ is one half of infinitely many almost symmetric
numerical semigroups $T$, with $t(T)=x$.
\end{cor}

\begin{proof}
The thesis is a direct consequence of the fact that $S$ is one
half of $S\!\Join^b\!\! E$ and of the previous corollary.
\end{proof}
%\medskip

\noindent \textbf{Acknowledgments.} The authors thank Fernando
Torres for pointing out the connection between numerical
duplication and Weierstrass semigroup theory.

\end{document}